\documentclass[reqno,12pt]{amsart}
\usepackage[english]{babel}
\usepackage{amssymb,amsthm,enumerate,xcolor,geometry,enumitem}

\normalfont
      plus .8\fontdimen3\font
     minus .8\fontdimen4\font

\newcommand{\assign}{:=}
\newcommand{\subRn}{\mathbb{R}^n}
\newcommand{\Pp}{\mathcal{P}}
\newcommand{\cdummy}{\cdot}
\newcommand{\pp}{p(\cdot)}
\newcommand{\qq}{q(\cdot)}
\newcommand{\tmSep}{; }

\newcommand{\tmop}[1]{\ensuremath{\operatorname{#1}}}

\newcommand{\Lp}{L^{p(\cdot)}}
\newcommand{\essinf}[1]{\underset{#1}{\tmop{ess} \, \inf \hspace{0.1em}}}
\newcommand{\esssup}[1]{\underset{#1}{\tmop{ess} \, \sup \hspace{0.1em}}}
\newcommand{\R}{\mathbb{R}}
\newcommand{\N}{\mathbb{N}}
\newcommand{\tmsep}{, }
\newcommand{\tmstrong}[1]{\textbf{#1}}
\newcommand{\tmtextit}[1]{{\itshape{#1}}}
\newcommand{\tmtextmd}[1]{{\mdseries{#1}}}
\newcommand{\tmtextup}[1]{{\upshape{#1}}}

\newtheorem{theorem}{Theorem}[section]
\newcommand{\asmpt}{{\Bumpeq}}

\newtheorem{lemma}[theorem]{Lemma}
\newtheorem{cor}[theorem]{Corollary}
\newtheorem{prop}[theorem]{Proposition}

\theoremstyle{definition}
\newtheorem{definition}[theorem]{Definition}

\theoremstyle{definition}
\newtheorem{example}[theorem]{Example}

\theoremstyle{remark}
\newtheorem{remark}[theorem]{Remark}

\numberwithin{equation}{section}

\newcommand{\qqq}{q (\cdot)}

\begin{document}
\title{Modular inequalities for the maximal operator in variable Lebesgue
spaces}
\author{David Cruz-Uribe, OFS}
\address{David Cruz-Uribe, OFS \\
Department of Mathematics \\
 University of Alabama \\ Tuscaloosa \\ AL 35487, USA}
\email{dcruzuribe@ua.edu}
\author{Giovanni Di Fratta}
\address{Giovanni Di Fratta
\\Institute for Analysis and Scientific Computing
\\TU Wien
\\Wiedner Hauptstraße 8-10, 1040 Wien, Austria}
\email{giovanni.difratta@asc.tuwien.ac.at}
\author{Alberto Fiorenza}
\address{Alberto Fiorenza \\
Dipartimento di
Architettura \\
Universit\`a di Napoli \\ Via Monteoliveto, 3 \\
I-80134 Napoli, Italy\\
and Istituto per le Applicazioni del Calcolo
``Mauro Picone", sezione di Napoli \\
Consiglio Nazionale delle Ricerche \\
via Pietro Castellino, 111 \\
 I-80131 Napoli, Italy }
\email{fiorenza@unina.it}

\begin{abstract}
A now classical result in the theory of variable Lebesgue spaces due
to Lerner~\cite{lerner2005modular} is
that a modular inequality for the
  Hardy-Littlewood maximal function in $L^{\pp}(\R^n)$ 
  holds if and only if the exponent is constant.   We generalize this
  result and give a new and simpler proof.    We then find 
  necessary and sufficient conditions  for the validity of the weaker
  modular inequality
\[ \int_\Omega Mf(x)^{p(x)}\,dx \
\leq 
c_1 \int_\Omega |f(x)|^{q(x)}\,dx + c_2, \]
where $c_1,\,c_2$ are non-negative constants and $\Omega$ is any
measurable subset of $\R^n$.  As a corollary we get sufficient conditions for the
modular inequality
\[ \int_\Omega |Tf(x)|^{p(x)}\,dx \
\leq 
c_1 \int_\Omega |f(x)|^{q(x)}\,dx + c_2, \]
where $T$ is any operator that is bounded on $L^p(\Omega)$, $1<p<\infty$.
\end{abstract}



\thanks{The first author is supported by NSF Grant 1362425 and
  research funds from the Dean of the College of Arts \& Sciences, the
  University of Alabama.}
\thanks{The second author acknowledges support through the special research program \emph{Taming complexity in partial differential systems} funded by the Austrian Science Fund (FWF) under grant F65 and of the Vienna Science and Technology Fund (WWTF) through the research project \emph{Thermally controlled magnetization dynamics} (grant MA14-44).}

{\maketitle}

\section{Introduction}
The variable Lebesgue spaces are a generalization of the classical Lebesgue
spaces, where the constant exponent $p$ is replaced by a variable exponent
function $\pp$. They have been studied extensively for the
past twenty years, particularly for their applications to PDEs, the calculus
of variations 
{\cite{acerbi2002regularity,cruz2013variable,diening2011lebesgue}}, but also
for their use in a variety of physical and engineering contexts: the modeling
of electrorheological fluids {\cite{ruzicka2000electrorheological}}, the
analysis of quasi-Newtonian fluids {\cite{zhikov1997meyers}}, fluid flow in
porous media {\cite{amaziane2009nonlinear}}, magnetostatics
{\cite{cekic2012lp}} and image reconstruction {\cite{blomgren1997total}}.

Let $\Omega \subset \R^n$ be a Lebesgue measurable set, $0<| \Omega |\le \infty$.
Given a measurable exponent function $\pp : \Omega \rightarrow [1,
\infty)$, hereafter denoted by $\pp\in\Pp
(\Omega)$,
for any measurable set $E \subset \R^n$, $|E \cap \Omega | > 0$, we
set
\begin{equation*}
  p_- (E) = \essinf{x \in E \cap \Omega}_{} p (x), \qquad p_+ (E) = \esssup{x
  \in E \cap \Omega} p (x) .
\end{equation*}
For brevity, we set $p_- = p_- (\Omega)$ and $p_+ = p_+ (\Omega)$. The space
$\Lp (\Omega)$ is defined as the set of all measurable functions $f$ such that
for some $\lambda > 0$, $\rho_{p (\cdot), \Omega} (f / \lambda) < \infty$,
where $\rho_{p (\cdot), \Omega}$ is the modular functional defined by
\begin{equation*}
  \rho_{p (\cdot), \Omega} (f) = \int_{\Omega} |f (x) |^{p (x)} 
  \hspace{0.17em} d_{} x. 
\end{equation*}
In situations where there is no ambiguity we will simply write $\rho_{p
(\cdot)} (f)$ or $\rho (f)$. The space $\Lp (\Omega)$ is a Banach function
space when equipped with the Luxemburg norm
\begin{equation}
  \|f\|_{L^{p (\cdot)} (\Omega)} = \inf \{\lambda > 0 \kern 2pt \colon \rho_{p (\cdot), \Omega}
  (f / \lambda) \leqslant 1\} . \label{eq:normdef}
\end{equation}
When $\pp = p$, a constant, then $\Lp (\Omega) = L^p (\Omega)$ and
{\eqref{eq:normdef}} reduces to the classical norm on $L^p (\Omega)$. For the
properties of these spaces, we refer the reader
to~{\cite{cruz2013variable,diening2011lebesgue}}.

\medskip

Given a function $f\in L^1_{loc}(\R^n)$, the
(uncentered) Hardy-Littlewood maximal function $M_{} f$ is defined for $x
\in \R^n$ by
\begin{equation*}
  M_{} f (x) = \sup_{Q \ni x}  \frac{1}{|Q|}  \int_Q | f (y) | d_{} y,
\end{equation*}
where the supremum is taken over all cubes $Q \subset \R^n$ containing $x$ and
whose sides are parallel to the coordinate
axes. (See~\cite{cruz2013variable,zuazo2001fourier}.)   If $f\in
L^1_{loc}(\Omega)$, then we define $Mf$ by extending $f$ to be
identically $0$ on $\R^n\setminus \Omega$. 
 The following result,
proved by Neugebauer and the first and  third authors~{\cite{cruz2002maximal,cruz2004corrections}},
gives a nearly optimal sufficient condition on the exponent $\pp$ for
the maximal operator to satisfy a norm inequality on $L^{\pp}(\Omega)$. 

\begin{theorem}
  \label{old-thm}Given an open set $\Omega\subset \R^n$, let
  $\pp\in\Pp (\Omega)$ be such that $1<p_- \le p_+<\infty$ and
  $\pp \in L_{} H (\Omega)$, i.e., $p (\cdot)$ is log-H{\"o}lder
  continuous both locally and at infinity:
  \begin{gather*}
    |p (x) - p (y) | \leqslant \frac{C_0}{- \log (|x - y|)}, \qquad \, |x - y| <
    \frac{1}{2}, \quad x,y\in\Omega,\\
    |p (x) - p_\infty| \leqslant \frac{C_{\infty}}{\log (e + |x|)},
    \qquad  x\in\Omega \, .
  \end{gather*}
  Then $M$ is bounded on
  $L^{p (\cdummy)} (\Omega)$:
  \begin{equation}
    \label{normineq} \|M_{} f\|_{L^{\pp}(\Omega)} \leqslant C \|f\|_{L^{\pp}(\Omega)} .
  \end{equation}
\end{theorem}

In the constant exponent case, Theorem~\ref{old-thm} reduces to the
classical result that the maximal operator is bounded on
$L^p(\Omega)$, $1<p<\infty$.  In this case, the norm inequality is
equivalent to the modular inequality
\[ \int_\Omega Mf(x)^p\,dx \leq C\int_\Omega |f(x)|^{p}\,dx. \]
Similar modular inequalities hold in the scale of Orlicz spaces:  see,
for instance,~\cite{kokilashvili1991weighted}.
It is therefore natural to consider the analogous question of modular
inequalities for the maximal operator on the variable Lebesgue spaces:
\begin{equation}
  \int_{\Omega} M_{} f (x)^{p (x)}  \hspace{0.17em} d_{} x \, \leqslant \, C
  \int_{\Omega} |f (x) |^{p (x)}  \hspace{0.17em} d_{} x. \label{var-modular}
\end{equation}
Since inequality~\eqref{var-modular} implies the norm
inequality~\eqref{normineq}, it is clear that stronger hypotheses may
be needed on the exponent function $\pp$ for the modular inequality to
hold.  The following example from~\cite{cruz2003hardy} shows that
log-H\"older continuity is not sufficient and the modular inequality
need not hold even for a smooth exponent function.

\begin{example}\label{primoesempio}
  Let $\pp\in\Pp
(\R)$ be a measurable exponent function
  which is equal to $2$ on the interval $[0, 1]$ and equal to $3$ on $[2, 3]$
  (we make no other assumptions on $p (\cdummy)$). Define the sequence of
  functions $\{f_k\}_{k \in \N} = \{k \chi_{[0, 1]}\}_{k \in \N}$.  Then for any $x \in
  [2, 3]$, 
  \begin{equation*}
  M_{} f_k (x) \, \geqslant \, \frac{1}{3}  \int_0^3 \left|f_k (y) \right|\kern 2pt d y \,=\,  \frac{k}{3},
  \end{equation*}
  so that 
  \begin{equation*}
  \rho_{p (\cdot), \R} (M_{} f_k)\, \geqslant \,\int_2^3 \left(\frac{k}{3}\right)^3 \kern 1pt d x \, =\, \frac{k^3}{27}\,.
   \end{equation*}
  On the other hand $\rho_{p (\cdot), \R} (f_k) = k^2$,
so {\eqref{var-modular}} cannot hold.
\end{example}

\medskip

In fact,  when $\Omega=\R^n$ and $p_+<\infty$, Lerner~\cite{lerner2005modular}
showed that inequality {\eqref{var-modular}} never holds unless $\pp$
is constant.

\begin{theorem}
  \label{thm:lerner-modular}Let $\pp\in\Pp
(\subRn)$, $p_+<\infty$. Then the  modular inequality
 \[
    \int_{\subRn} M_{} f (x)^{p (x)}  \hspace{0.17em} d_{} x \, \leqslant \, C_{\pp,n}
    \int_{\subRn} |f (x) |^{p (x)}  \hspace{0.17em} d_{} x,
\]
  where $C_{\pp,n}$ is a constant depending on $n, p (\cdot)$ but independent of $f$,
  holds if and only if there is a constant $p>1$ such that $\pp = p$
  almost everywhere.
\end{theorem}

\begin{remark}
The original proof of Theorem~\ref{thm:lerner-modular}
in~\cite{lerner2005modular} (see
also~\cite[Theorem~3.31]{cruz2013variable}) used the theory of
Muckenhoupt $A_p$ weights from harmonic analysis.  For a simpler
proof, see~\cite{izuki2013} and Corollary~\ref{Cor:Lernergen} below.
\end{remark}

\medskip

However, weaker modular inequalities that include an error term are
true.  These results played a role in the original proofs of
Theorem~\ref{old-thm}.   For
instance, we have the following result {\cite[Theorem 3.33]{cruz2013variable}}. 

\begin{theorem}
  \label{thm:modular-log}Given $\pp \in \Pp (\R^n)$ such that
  $1<p_-\leq p_+<\infty$ and $\pp \in L_{} H
  (\R^n)$, suppose $f \in \Lp (\R^n)$ and $\|f\|_{\pp} \leqslant 1$. 
  Then
  \begin{equation*}
    \label{eqn:mod-log1} \int_{\subRn} M_{} f (x)^{p (x)}  \hspace{0.17em}
    d_{} x \, \leqslant \,  C _{\pp,n}\int_{\subRn} |f (x) |^{p (x)}  \hspace{0.17em} d_{} x + C_{\pp,n}
    \int_{\subRn} \frac{d_{} x}{(e + |x|)^{np_-}},
  \end{equation*}
where the constant $C_{\pp,n}$  depends  on $n, p (\cdot)$ but is independent of $f$.
\end{theorem}

\medskip

The goal of this paper is to give necessary and sufficient conditions
for modular inequalities of the form
\begin{equation}
  \int_{\Omega} M_{} f (x)^{p (x)} d_{} x \leqslant c_1  \int_{\Omega} |f (x) |^{q
  (x)} d_{} x + c_2, \label{mainmodular}
\end{equation}
to hold for all measurable functions $f$, where $\pp,\,\qq \in \Pp
(\Omega)$, and $c_1 > 0$, $c_2 \ge 0$ are constants depending on $n,
\pp, \qq$ and $|\Omega |$, but are independent of $f$.  We are
interested in the weakest possible conditions on the exponent
functions $\pp$ and $\qq$ for \eqref{mainmodular} to hold.  In
particular, we want to prove modular inequalities without assuming any smoothness
conditions on the exponents.  

In this paper we will only consider the case $p (\cdot)\not\equiv 1$.
The endpoint case when $p (\cdot)\equiv 1$ is substantially different.
If $\Omega$ is bounded and $q_->1$, then \eqref{mainmodular} always
holds: this is an immediate consequence of
\cite[Theorem~1.2]{CFtransAMS}.  If $\Omega=\R^n$,
then~\eqref{mainmodular} never holds, since $Mf$ is never in
$L^1(\R^n)$ unless $f=0$ a.e.  More generally, given any set $\Omega$
with infinite measure, then arguing as in
Example~\ref{example:infinite} below, we would have $L^{\qq}(\Omega)
\subset L^1(\Omega)$, which is impossible:
see~\cite[Theorem~2.45]{cruz2013variable}. 
When $q_-=1$ the problem of
characterizing $\qq$ is open.  Some delicate results
in~\cite{futamuramizuta, hastoappr1} show that this problem depends on
how quickly $q(\cdot)$ approaches $1$.

\medskip

Our two main results completely characterize the exponents $\pp$ and
$\qq$ so that the modular inequality holds.  Our characterization
depends strongly on whether $\Omega$ has finite or infinite measure;
When $\Omega$ has finite measure our result is remarkably simple.

\begin{theorem}\label{primissimo}
  Given a set $\Omega \subseteq \R^n$, $0<| \Omega | <\infty$,  let $p(\cdot), \qqq \in \Pp (\Omega)$, $p (\cdot)\not\equiv 1$. Then the modular inequality
  {\eqref{mainmodular}} holds if and only if $p_+ (\Omega) \leqslant q_-
  (\Omega)$.
\end{theorem}

As our second result below shows, the assumption that $|\Omega|<\infty$ is
critical in Theorem~\ref{primissimo}.  But to motivate this result, we
first give the following example.

\begin{example} \label{example:infinite}
If $\Omega \subseteq \R^n$, $| \Omega | = \infty$, and if $\pp \in \Pp (\Omega)$, $\qqq \in \Pp (\Omega)$, then
the assumption that $p_+ (\Omega) \leq q_- (\Omega)$
is not sufficient for {\eqref{mainmodular}} to be true.
We first consider the case $p_+(\Omega)=q_-(\Omega)$.  Fix an open set
$\Omega$, $|\Omega|=\infty$, and constants
$1<p<q<\infty$. Define  $p (\cdot) \equiv p $ and
\[ q (x) = \left\{ \begin{array}{ll}
     p & \text{if } x \in Q\\
     q & \text{if } x \in \Omega \backslash Q,
   \end{array} \right. \]
where $Q\subset \Omega$ is a cube.  Then $p_+ (\Omega) = q_-
(\Omega)$.  Suppose
{\eqref{mainmodular}}  holds; then we would have
\begin{align*}
  \int_{\Omega} | f (x) |^p  \hspace{0.17em} d_{} x & \leqslant 
  \int_{\Omega} M_{} f (x)^p  \hspace{0.17em} d_{} x \\
  & \leqslant  c_1 \int_Q | f (x) |^p  \hspace{0.17em} d_{} x + c_1
  \int_{\Omega \backslash Q} | f (x) |^q  \hspace{0.17em} d_{} x + c_2 .
\end{align*}
But then, if we let $f \assign g 
\chi_{\Omega \backslash Q}$, we would get the
embedding $^{} L^q (\Omega \backslash Q) \subset L^p (\Omega \backslash Q)$,
which does not hold when $p < q$ since $\Omega$ has infinite measure
{\cite{kabaila1981inclusion,villani1985another}}. 

The case $p_+ (\Omega) < q_-
(\Omega)$ is obtained from the same argument by taking $Q = \emptyset$.
\end{example}

\medskip

The problem in Example~\ref{example:infinite} arises because the
exponents $\pp$ and $\qq$ behave  differently at
infinity.    To avoid this, we make the following definition.

\begin{definition}
Given a set $\Omega$, $|\Omega|=\infty$, let $\mathcal{F}_{\Omega}$
denote the collection of subsets of  $\Omega$ that havie infinite
measure.   Given $p (\cdot), q (\cdot) \in \mathcal{P} (\Omega)$, we
say that $\pp$ and $\qq$ {\tmstrong{touch at infinity}}, and denote this
by $p (\cdot) \asmpt q (\cdot)$, if for every $E \in \mathcal{F}_{\Omega}$,
\[ p_+ (E) = p_+ (\Omega) = q_- (\Omega) =  q_- (E). \]
\end{definition}

The exponents in Example~\ref{example:infinite} do not touch at
infinity.  We consider three additional examples.

\begin{example}
  Let $\Omega=\R$.  
\begin{enumerate}

\item The exponents $p(x)=2-(1+x^2)^{-1}$,
  $q(x)=2+(1+x^2)^{-1}$ touch at infinity.  

\item On the other hand, if we
  let 
  $ \tilde q(x)=a+(1+x^2)^{-1}$, $a>2$, then $p (\cdot)$ and
  $\tilde q (\cdot)$ do not touch at infinity.   

\item Finally, if $p(x)\equiv 2$ and
  $q(x)=2+\chi_{E}$, where $E$ is any bounded measurable set, then
  $\pp$ and $\qq$ touch at infinity.
\end{enumerate}
\end{example}

We can now state our second main result, characterizing the modular
inequality on sets $\Omega$ with infinite measure.

\begin{theorem} \label{main} 
Given a set $\Omega\subseteq \R^n$, $|\Omega|=\infty$,
  let $p (\cdot), \qqq \in \Pp (\Omega)$, $p (\cdot)\not\equiv
  1$. Define
  $D \assign \{ x \in \Omega : p (x) < q (x) \} \neq \emptyset$. 
Then the
  following  are equivalent:

\begin{description}[align=right,labelwidth=1.1cm]
\item[\quad\tmtextmd{\tmtextup{$(i)$}}] The modular inequality {\eqref{mainmodular}} holds;
\medskip
\item[\quad\tmtextmd{\tmtextup{$(i_{}
    i)$}}] $p (\cdot) \asmpt q (\cdot)$ and $L^{q (\cdot)} (\Omega)
    \hookrightarrow L^{p (\cdot)} (\Omega)$;
\medskip
\item[\quad\tmtextmd{\tmtextup{$(i_{} i_{}
    i)$}}] $p (\cdot) \asmpt q (\cdot)$ and there exists  $\lambda > 1$
    such that
    \begin{equation}
      \rho_{r (\cdot), D} (1 / \lambda) = \int_D \lambda^{- r (x)} d_{} x <
      \infty, \label{eq:condsuffmaximal1}
    \end{equation}
where $r (\cdot)$ is the defect exponent defined by $\frac{1}{r (x)} =
    \frac{1}{p (x)} - \frac{1}{q (x)}$;
\medskip 
 \item[\quad\tmtextmd{\tmtextup{$(i_{}
    v)$}}] $p (\cdot) \asmpt q (\cdot)$ and there exists a measurable
    function $\omega$, $0 < \omega (\cdot) \leqslant 1$, such that
    \begin{equation}
      \rho_{p (\cdot), D} (\omega) = \int_D \omega (x)^{p (x)} d_{} x < \infty
      \label{eq:condsonomega}
    \end{equation}
    and
    \begin{equation}
      \| \omega (\cdot)^{- | p_+ - p (\cdot) |} \|_{L^{\infty} (D)} \cdot \| \omega (\cdot)^{-
      | q (\cdot) - p_+ |} \|_{L^{\infty} (D)} < \infty .
      \label{eq:condsonomega2}
    \end{equation}
\end{description}
\end{theorem}

\begin{remark}   \label{rmk:a2impla3}
There is a close connection between the embedding $L^{q (\cdot)} (\Omega)
    \hookrightarrow L^{p (\cdot)} (\Omega)$ and
    condition~\eqref{eq:condsuffmaximal1}:  we have that this embedding
    holds if and only if $p(x)\leq q(x)$
    a.e. and~\eqref{eq:condsuffmaximal1} holds  (see \cite[Theorem
    2.45]{cruz2013variable}).   However, \eqref{eq:condsuffmaximal1} is independent
    of $p (\cdot) \asmpt q (\cdot)$.   For one direction, let
    $\Omega=(2,\infty)$ and define $\pp$ and $\qq$ by
\[  \frac{1}{p(x)} = \frac{1}{2}-\frac{1}{x^2}, \qquad \frac{1}{q(x)}
  = \frac{1}{p(x)} - \frac{1}{x^4}. \]
Then we have $p(x)\leq q(x)$ and the defect exponent is $r(x)=x^4$,
so~\eqref{eq:condsuffmaximal1} holds  for any $\lambda>1$.   Thus $L^{q (\cdot)} (\Omega)
    \hookrightarrow L^{p (\cdot)} (\Omega)$.  However, $p_+(\Omega)=4$
    and $q_-(\Omega)=2$, so we do not have that $\pp$ and $\qq$ touch
    at infinity.

Conversely, let $\Omega=(e^9,\infty)$ and define $\pp$ and $\qq$
by 
\begin{equation*}
p(x)=2, \qquad q(x)=\frac{2\log\log(x)}{(\log\log (x)-2)}. 
    \end{equation*}
Then $q_-(\Omega)=2$ and it decreases to this value as $x\rightarrow
\infty$, so $p (\cdot) \asmpt q (\cdot)$.   However, the defect
exponent is   $r(x)=\log\log (x)$ and the integral
in~\eqref{eq:condsuffmaximal1} is infinite for any value of $\lambda>1$. 
    \end{remark}

\begin{remark}
The condition~\eqref{eq:condsuffmaximal1} is closely related to the
problem of finding sufficient conditions for the maximal operator to
be bounded on $L^{\pp}(\Omega)$ when $\Omega$ is unbounded.
Nekvinda~\cite{MR2057644} 
showed that the log-H\"older continuity condition at infinity in Theorem~\ref{old-thm}
can be replaced by a weaker integral condition: for some $\lambda>1$,
\[  \int_D \lambda^{-r(x)}\,dx <\infty, \]
where now  $r(\cdot)$ is defined by $\frac{1}{r(x)}= \left|
    \frac{1}{p(x)}-\frac{1}{p_\infty}\right|$ and $D=\{ x \in \Omega :
  p(x) \neq p_\infty \}$.   For a thorough discussion of this
  condition and its relationship with~\eqref{eq:condsuffmaximal1} and the
  associated embedding theorem,
  see~\cite[Section~4.1]{cruz2013variable}.
\end{remark}


\begin{remark}
As a consequence of the assumption that $p (\cdot)
  \asmpt q (\cdot)$, we have that for any $R>0$, 
\[ p_+(\Omega\setminus B(0,R)) = p_+(\Omega) = q_-(\Omega) =
  q_-(\Omega\setminus B(0,R)). \]
Therefore, $\pp$ and $\qq$ have a common asymptotic value in the sense
that
\[ p_+(\Omega) = \limsup_{|x|\rightarrow \infty} p(x) 
= \liminf_{|x|\rightarrow \infty} = q_-(\Omega).  \]
Denote this asymptotic value by $p_\infty$; it is a generalization of
the value $p_\infty$ that occurs in the definition of log-H\"older
continuity in Theorem~\ref{old-thm} or in the Nekvinda condition
discussed above.   In particular, 
  if  $| D | = \infty$, then condition~{\eqref{eq:condsonomega2}} is
  equivalent to
  \begin{equation*}
    \| \omega (\cdot)^{- | p_{\infty} - p (\cdot) |} \|_{L^{\infty} (D)} \cdot \|
    \omega (\cdot)^{- | q (\cdot) - p_{\infty} |} \|_{L^{\infty} (D)} < \infty.
  \end{equation*}
\end{remark}

\begin{remark}
It is a consequence of the proof that if the modular
inequality~\eqref{mainmodular} holds, then $c_1 \geqslant 1$: see
  the proof of the implication $(i) \Rightarrow (i_{} i)$.
\end{remark}

\medskip

In the proof of Theorems~\ref{primissimo} and~\ref{main}, we  use
the definition of the maximal operator to prove necessity.  In the
proof of sufficiency, we only use the fact that the maximal operator
is a bounded operator on $L^p(\Omega)$, $1<p<\infty$.  Therefore, as
an immediate corollary of the proofs we get the following result.

\begin{cor} \label{cor:sio}
Given a set $\Omega$ and $\pp,\,\qq \in \Pp(\Omega)$, suppose that
either $|\Omega|<\infty$ and $p_+(\Omega)\leq q_-(\Omega)$, or
$|\Omega|=\infty$,  $p (\cdot)  \asmpt q (\cdot)$,
and~\eqref{eq:condsuffmaximal1} holds.   If $T$ is
any operator that is bounded on $L^p(\Omega)$ for all $1<p<\infty$, then
\[ \int_\Omega |Tf(x)|^{p(x)}\,dx \leq c_1\int_\Omega
  |f(x)|^{q(x)}\,dx +c_2, \]
with positive constants $c_1,\,c_2$ that depend on $\pp$, $\qq$ and
$T$ but not on $f$. 
\end{cor}

The assumption on the operator $T$ is very general and is satisfied by
most of the classical operators of harmonic analysis:  for example, it
holds for Calder\'on-Zygmund singular integral operators and square
functions.   In fact, a close examination of the proof shows that we
can assume less:  given fixed $\pp$ and $\qq$, we only require that the
operator is bounded on $L^{p_+}(\Omega)$.   As a consequence, we can
prove a
modular inequality for the Fourier transform 
\[ \hat{f}(\xi) = \int_{\subRn} f(x)e^{-2\pi i x\cdot \xi}\,dx \]
on variable Lebesgue spaces, using the Plancherel theorem that
$\|\hat{f}\|_2=\|f\|_2$.  The importance of this result follows from
the fact that natural generalization of the Hausdorff-Young inequality
fails in the variable exponent setting.
(See~\cite[Section~5.6.10]{cruz2013variable} for complete details.)

\begin{cor}
Given $\pp, \, \qq \in \Pp(\R^n)$, $p_+=2$, suppose $p (\cdot)  \asmpt q (\cdot)$,
and~\eqref{eq:condsuffmaximal1} holds.  Then 
\[ \int_{\subRn} |\hat{f}(\xi)|^{p(\xi)}\,d\xi \leq c_1\int_{\subRn}
  |f(x)|^{q(x)}\,dx +c_2, \]
with positive constants $c_1,\,c_2$ that depend on $\pp$ and $\qq$ but
not on $f$. 
\end{cor}

\begin{remark}
  Modular inequalities for other operators that are bounded on
  $L^p(\Omega)$ have been extensively studied in the setting of Orlicz
  spaces: see, for example, {\cite{bongioanni2003modular,
      capone1995maximal, carro2000modular, carro2003some,
      fusco1990higher, kokilashvili1991weighted}}.  Modular
  inequalities in the variable Lebesgue spaces for operators other
  than the maximal operator have not been studied, though we refer the
  reader to~\cite{giannetti2011modular} for a modular interpolation
  inequality in variable Sobolev spaces.
\end{remark}

\begin{remark}
There is a certain parallel between Corollary~\ref{cor:sio} and the
theory of Rubio de Francia extrapolation in the scale of variable
Lebesgue spaces.  Roughly speaking, the theory of extrapolation says
that if an operator $T$ is bounded on $L^p(w)$, where $1<p<\infty$ and
$w$ is any weight in the Muckenhoupt $A_p$ class, then $T$ is bounded
on $L^{\pp}$ provided that the maximal operator is bounded on
$L^{\pp}$.  (See~\cite[Section~5.4]{cruz2013variable} for a precise
statement of extrapolation.)  We can restate Corollary~\ref{cor:sio}
as saying that if $T$ is bounded on $L^p(\Omega)$, and the maximal
operator satisfies a certain modular inequality, then (because
Theorems~\ref{primissimo} and~\ref{main} given necessary as well as
sufficient conditions) $T$ satisfies the same modular inequality
(possibly with different constants). 
\end{remark}

\begin{remark}
It would be of interest to generalize Corollary~\ref{cor:sio} and
modular inequalities on Orlicz spaces by considering the analogous
question in the scale of Musielak-Orlicz spaces~\cite{MR724434}.  It
would also be interesting to determine if the conditions in
Corollary~\ref{cor:sio} are necessary for any other 
operators to satisfy a modular inequality.
\end{remark}

\medskip

If we consider constant exponent functions $\pp=p$ and $\qq=q$, then
Theorems~\ref{primissimo} and~\ref{main}, and Corollary~\ref{cor:sio} have the following corollary.

\begin{cor}  \label{cormain}
Given $\Omega \subset \R^n$, suppose $| \Omega | <\infty$. If $1 < p \leqslant q < \infty$, then the following
  inequality holds
  \begin{equation}
    \int_{\Omega} M_{} f (x)^p  \hspace{0.17em} d_{} x \leqslant c_1 
    \int_{\Omega} |f (x) |^q  \hspace{0.17em} d_{} x + c_2,
    \label{eq:modconst}
  \end{equation}
  for every $f \in L^q (\Omega)$ and for some positive constants
  $c_1, c_2$ depending on $n, p, q, | \Omega |$, but independent of
  $f$. If $|\Omega|=\infty$, then inequality~\eqref{eq:modconst} holds
  if and only if $1 < p = q$.

Moreover, if $T$ is an operator that is bounded on $L^p(\Omega)$,
$1<p<\infty$, then these conditions are sufficient for $T$ to satisfy
the modular inequality
\[ \int_\Omega |Tf(x)|^p\,dx \leq c_1\int_\Omega |f(x)|^q\,dx +
  c_2. \]
\end{cor}

\medskip

To prove Theorems~\ref{primissimo} and~\ref{main}, we will first prove
the following proposition which establishes a necessary condition
which for sets $\Omega$ of finite measure is also sufficient.

\begin{prop}  \label{thm:maintheorem}
Given $\pp,\, \qq \in \Pp (\Omega)$, if the modular
  inequality {\eqref{mainmodular}} holds, then 
  \begin{equation}
    p_+ (\Omega) \le q_- (\Omega) \hspace{0.17em} .
    \label{eq:mainconditionnew}
  \end{equation}
\end{prop}

As a corollary to Proposition~\ref{thm:maintheorem}, together with the
classical theorem on the boundedness of the maximal operator on
$L^p(\Omega)$, $1<p<\infty$ (cf.~\cite{stein1971singular}), we
immediately get the following generalization of
Theorem \ref{thm:lerner-modular} to arbitrary domains and unbounded
exponent functions.  

\begin{cor}
  \label{Cor:Lernergen}
Given an open set $\Omega$ and $\pp \in
  \Pp (\Omega)$,  the modular inequality
  \[ \int_{\Omega} M_{} f (x)^{p (x)}  \hspace{0.17em} d_{} x \leqslant c_1 
     \int_{\Omega} |f (x) |^{p (x)}  \hspace{0.17em} d_{} x + c_2, \]
  with positive constants $c_1, c_2$ depending on $n, \pp, \qq$ and  $| \Omega |$ but
  independent of $f$, holds if and only if $\pp$ equals a constant
  $p>1$ almost everywhere.
\end{cor}

\begin{remark}
  Theorem \ref{Cor:Lernergen} does not
  contradict Theorem~\ref{thm:modular-log}, since in the latter result
  we need the additional hypothesis that $\|f\|_{L^{\pp}(\R^n)}\leq 1$. 
\end{remark}

\medskip

The remainder of this paper is organized as follows.  In
Section~\ref{section:prop} we first prove Proposition~\ref{thm:maintheorem}.
In Section~\ref{section:finite} we prove Theorem~\ref{primissimo} and
in Section~\ref{section:infinite} we prove Theorem~\ref{main}.

\section{Proof of Proposition \ref{thm:maintheorem}}
\label{section:prop}

We begin with a definition and a lemma.  Given a measurable set
$\Omega\subseteq \R^n$, $|\Omega|>0$, we denote
by $\mathcal{Q}_{\Omega}$ the set of open cubes $Q$ in $\R^n$ (whose sides are
parallel to the coordinate axes) such that $|\Omega\cap Q|>0$.

\begin{lemma} \label{Lemma:equivalence}
Given a set $\Omega\subseteq \R^n$, let $\pp \in \Pp (\Omega)$, $\qqq
\in \Pp (\Omega)$. Then the following conditions are equivalent:
\begin{description}[align=right,labelwidth=1.2cm]
\item[\tmtextmd{\tmtextup{$(i)$}}] $p_+ (Q) \leqslant q_- (Q)$ for every $Q \in
    \mathcal{Q}_{\Omega}$;
\medskip    
   \item[\tmtextmd{\tmtextup{$(i_{} i)$}}]$p_+ (\Omega) \leqslant q_- (\Omega)$.
\end{description}  
\end{lemma}
  

\begin{proof}
  The fact that $(i_{} i)$ implies $(i)$ is easy: for any $Q \in \mathcal{Q}_{\Omega}$
  we have
  \[ p_+ (Q) = p_+ (Q \cap \Omega) \leqslant p_+ (\Omega) \leqslant q_-
     (\Omega) \leqslant q_- (Q \cap \Omega) = q_- (Q) . \]
 
 In order to prove that $(i)$ implies $(i_{} i)$, let $\{Q_n\}_{n \in
   \N}$ be a countable cover of $\Omega$  by elements of
  $\mathcal{Q}_{\Omega}$. We then have that if $p_+ (Q) \le q_- (Q)$ for
  every $Q \in \mathcal{Q}_{\Omega}$, then
  \begin{equation}
    p_+ (Q_m) \le q_- (Q_n) \qquad \forall m, n \in \N . \label{eq:fundobs}
  \end{equation}
To see this, note that for every $m, n \in \N$, there exists a cube $Q_{m, n} \in
  \mathcal{Q}_{\Omega}$ such that $Q_m \cup Q_n \subseteq Q_{m, n}$.  By
  hypothesis $p_+ (Q_{m, n}) \le q_- (Q_{m, n})$, so
  $$p_+ (Q_m) \leqslant
  p_+ (Q_{m, n}) \leqslant q_- (Q_{m, n}) \leqslant q_- (Q_n)\, .$$
  
  Now, if we first take the supremum over $m \in \N$ and then take the
  infimum over $n \in \N$, by~\eqref{eq:fundobs} we get
  $\sup_{m \in \N} p_+ (Q_m) \leqslant \inf_{n \in \N} q_- (Q_n)$.
  Therefore,
  \begin{align*}
    p_+ (\Omega) = p_+ \bigg( \underset{m \in \N}{\bigcup} Q_m \bigg) & =
    \sup_{m \in \N} p_+ (Q_m) \\
		&\leqslant \inf_{n \in \N} q_- (Q_n) = q_- \bigg(
    \underset{n \in \N}{\bigcup} Q_n \bigg) = q_- (\Omega) .
  \end{align*}
\end{proof}

The following argument is inspired by Example~\ref{primoesempio} and
is similar to the proof of Theorem~\ref{thm:lerner-modular}
in~\cite[Thm. 5.1]{izuki2013}.

\begin{proof}[Proof of Proposition \ref{thm:maintheorem}] 
If {\eqref{eq:mainconditionnew}} does not hold, then by Lemma
  \ref{Lemma:equivalence} there exists a cube $Q \in \mathcal{Q}_{\Omega}$
  such that $p_+ (Q) > q_- (Q)$. Let $\alpha, \beta$ be such that
  \[ q_- (Q) < \alpha < \beta < p_+ (Q) \hspace{0.17em} . \]
  Let $E_{\beta} \subset Q \cap \Omega$, $|E_{\beta} | > 0$, be such that $p
  (x) \ge \beta$ for a.e. $x \in E_{\beta}$. Similarly, let $E_{\alpha}
  \subset Q \cap \Omega$, $|E_{\alpha} | > 0$, be such that $q (x) \le \alpha$
  for a.e. $x \in E_{\alpha}$.  Define $f = \lambda \chi_{E_{\alpha}}$,
  where $\lambda > 1$. Then for all $z\in Q$,
  \[ M_{} f (z) \ge \frac{1}{| Q |} \int_Q |f (y) |  \hspace{0.17em} d_{} y =
     \frac{\lambda |E_{\alpha} |}{|Q|}. \]
  Moreover, if $\lambda > | Q | / | E_{\alpha} |$, then $(\lambda |E_{\alpha} |
  / |Q|)^{p (x)} \ge (\lambda |E_{\alpha} | / | Q |)^{\beta}$ for every $x \in
  E_{\beta}$.  Hence,
  \[ \int_{\Omega} M_{} f (x)^{p (x)} d_{} x \geqslant \int_{E_{\beta}} \left(
     \frac{\lambda |E_{\alpha} |}{|Q|} \right)^{p (x)} d_{} x \geqslant
     |E_{\beta} | \left( \frac{\lambda |E_{\alpha} |}{|Q|} \right)^{\beta} .
  \]
  On the other hand,
  \[ \int_{\Omega} |f (x) |^{q (x)}  \hspace{0.17em} d_{} x =
     \int_{E_{\alpha}} \lambda^{q (x)}  \hspace{0.17em} d_{} x \le |E_{\alpha}
     | \lambda^{\alpha}. \]
 Therefore, if {\eqref{mainmodular}} holds, then we must have that
  \[ |E_{\beta} | \left( \frac{\lambda |E_{\alpha} |}{|Q|} \right)^{\beta} \le
     c_1 |E_{\alpha} | \lambda^{\alpha} + c_2 \]
  for all $\lambda$ sufficiently large, which is a contradiction since $\alpha <
  \beta$.
\end{proof}

\section{Proof of Theorem~\ref{primissimo}}
\label{section:finite}

By Proposition~\ref{thm:maintheorem} we have that if the modular
inequality~\eqref{mainmodular} holds, then $p_+(\Omega)\leq
q_-(\Omega)$.  Therefore, it remains to show that this condition is
sufficient. 

Fix a set $\Omega$ and $\pp,\,\qq \in \Pp(\Omega)$ such
that $p_+(\Omega)\leq q_-(\Omega)$, and fix a function $f$.   Given a
set $E\subseteq \Omega$, we define
\[ I(E) = \int_E Mf(x)^{p(x)}\,dx, \qquad F(E) = \int_E
  |f(x)|^{p_+}\,dx,  \]
and
\[ D_1(Mf) = \{ x\in \Omega : Mf(x)>1 \}, \qquad 
D_1(f) = \{ x\in \Omega : |f(x)|>1 \}. \]

We now estimate as follows:
\[ \int_\Omega Mf(x)\,dx = I(D_1(Mf)) + I(\Omega\setminus
  D_1(Mf)).  \]
We immediately have that $I(\Omega\setminus
  D_1(Mf)) \leq |\Omega|$.  On the other hand, since $\pp\not\equiv
  1$, $p_+>1$, so the maximal
  operator is bounded on $L^{p_+}(\Omega)$.  Hence,
\[  I(D_1(Mf))  \leq \int_{D_1(Mf)} Mf(x)^{p_+}\,dx \leq c_{p_+,n}
\int_\Omega |f(x)|^{p_+}\,dx = c_{p_+,n} F(\Omega).  \]

To estimate $F(\Omega)$ we argue similarly:  since $p_+(\Omega)\leq q_-(\Omega)$,
\[ F(\Omega) = F(D_1(f))+ F(\Omega\setminus D_1(f))
\leq \int_{D_1(f)} |f(x)|^{q(x)}\,dx + |\Omega|. \]
If we combine all of these inequalities, we get
\[ \int_\Omega Mf(x)\,dx \leq c_{p_+,n}\int_\Omega |f(x)|^{q(x)}\,dx +
(c_{p_+,n}+1)|\Omega|. \]
This completes the proof of sufficiency.

\section{Proof of Theorem \ref{main}}
\label{section:infinite}

We will prove the following chain of implications:
\[ 
(i) \Rightarrow (i_{} i) \Rightarrow (i_{} i_{} i) \Rightarrow (i_{} v)
     \Rightarrow (i) . 
\]

\medskip
\noindent $[(i) \Rightarrow (i_{} i)]$ We first prove that if the modular inequality
  {\eqref{mainmodular}} holds,  then $L^{q (\cdot)} (\Omega) \hookrightarrow
  L^{p (\cdot)} (\Omega)$.  Since $L^{p (\cdot)}$ is a Banach function space,
  the embedding $L^{q (\cdot)} (\Omega) \hookrightarrow L^{p (\cdot)}
  (\Omega)$ is equivalent (cf. {\cite[Thm. 1.8]{bennett1988interpolation}}) to
  the set-theoretical inclusion $L^{q (\cdot)} (\Omega) \subseteq L^{p
  (\cdot)} (\Omega)$.   Since $M_{} f (x) \geqslant | f (x) |$ a.e. in $\Omega$,
  if {\eqref{mainmodular}} holds, then
  $\rho_{p (\cdot), \Omega} (f) \leqslant c_1 \rho_{q (\cdot), \Omega} (f) +
  c_2$. 
Fix  $f \in L^{q (\cdot)} (\Omega)$;  then for some $\lambda > 0$, $\rho_{q (\cdot), \Omega} (f /
  \lambda) < \infty$.  Therefore,
  \begin{equation*}
    \rho_{p (\cdot), \Omega} (f / \lambda) \leqslant c_1 \rho_{q (\cdot),
    \Omega} (f / \lambda) + c_2 < \infty,
  \end{equation*}
and so  $f \in L^{p (\cdot)} (\Omega)$.
  
 We now prove that if {\eqref{mainmodular}}
  holds, then $p (\cdot) \asmpt q (\cdot)$.   Given  any measurable set $E \in
  \mathcal{F}_{\Omega}$ and any measurable function $f : E \subseteq \Omega
  \rightarrow \R$, \eqref{mainmodular} implies that
  \begin{equation}
    \int_E | f (x) |^{p (x)} d_{} x \leqslant c_1 \int_E | f (x) |^{q (x)}
    d_{} x + c_2, \label{eq:implmodularinfmeasure}
  \end{equation}
  with $c_1, c_2 > 0$ the same constants.
    Fix $E \in
  \mathcal{F}_{\Omega}$ and define $f (x) = \lambda \cdot \chi_{B_{\delta}\cap E} (x)$, $0 < \lambda <
  1$ and $B_{\delta}=B(0,\delta)$. 
Since $0< \lambda < 1$, for $x\in E$, $\lambda^{p_+ (E)} \leqslant
  \lambda^{p (x)}$ and $\lambda^{q (x)} \leqslant \lambda^{q_- (E)}$.
  Therefore, by {\eqref{eq:implmodularinfmeasure}}, 
  \begin{align*}
    | E \cap B_{\delta} | \, \lambda^{p_+ (E)} \; & \leqslant  \; \int_{E \cap
    B_{\delta}} \lambda^{p (x)} d_{} x  \\
    & \leqslant\; c_1 \int_{E \cap B_{\delta}}
    \lambda^{q (x)} d_{} x + c_2 
    \; \leqslant \; c_1 | E \cap B_{\delta} |
    \lambda^{q_- (E)} + c_2 . 
  \end{align*}
  Since $| E \cap B_{\delta} | \rightarrow \infty$ as $\delta \rightarrow
  \infty$, we get that $\lambda^{p_+ (E)} \leqslant c_1
  \lambda^{q_- (E)} + c_2 | E \cap B_{\delta} |^{- 1}$ for $\delta$
sufficiently large. If we take the limit as $\delta \rightarrow
  \infty$,  we get  that if {\eqref{mainmodular}} holds,
  then
  $$\lambda^{p_+ (E)} \leqslant c_1 \lambda^{q_- (E)}\qquad \forall\, 0<\lambda<1\, .$$ 
Since $p_+ (E) \leqslant q_-(E)$ we must have that $p_+ (E) = q_- (E)$ and $c_1 \geqslant 1$. 

Finally, since by Theorem \ref{thm:maintheorem},  $p_+ (\Omega) \leqslant q_-
  (\Omega)$,  and since
  $p_+ (E) \leqslant p_+ (\Omega) \leqslant q_- (\Omega) \leqslant q_- (E)$,
  we
  get that $p (\cdot) \asmpt q (\cdot)$.
  
\bigskip

  \noindent  $[(i_{} i) \Rightarrow (i_{} i_{} i)]$   As noted above,
  this implication follows from the fact that the embedding  $L^{q (\cdot)} (\Omega)
    \hookrightarrow L^{p (\cdot)} (\Omega)$ is equivalent to assuming
    $p(x)\leq q(x)$ and \eqref{eq:condsuffmaximal1} holds.  (See~\cite[Thm.
    2.45]{cruz2013variable}.)

  \bigskip

  \noindent $[(i_{} i_{} i) \Rightarrow (i_{} v)]$ We explicitly
  construct the function $\omega$.  Since
  $p (\cdot) \asmpt q (\cdot)$, we claim that there exists
  $\kappa > 1$ such that $| E_{q (\cdot), \kappa} | < \infty$, where
  $E_{\qq,\kappa}= \{ x \in \Omega : q(x)>\kappa \}$.  For if not,
  then for all $\kappa>1$, $| E_{q (\cdot), \kappa} | = \infty$.  In
  particular, if we set $\kappa=p_+(\Omega)+1$, then
  $E_{q (\cdot),\kappa} \in \mathcal{F}_\Omega$ and
  $q_-(E_{q (\cdot),\kappa}) >p_+(\Omega) \geq p_+(E_{q
    (\cdot),\kappa})$, a contradiction.

Fix such a $\kappa$ and define
  \begin{equation*}
    \omega (x) \assign \left\{ \begin{array}{ll}
      \lambda^{- r (x) / p (x)} & x\in  D \backslash E_{q (\cdot),
      \kappa},\\
      1 & x\in D \cap E_{q (\cdot), \kappa},
    \end{array} \right.
  \end{equation*}
  where $r (\cdot)$ is the defect exponent defined by $\frac{1}{r (x)} =
  \frac{1}{p (x)} - \frac{1}{q (x)}$.  Since $\lambda > 1$, 
  we have that $0 < \omega (\cdot) \leqslant 1$ and 
  \begin{align*}
    \omega (\cdot)^{- | p_+ - p (\cdot) |} \; = \; \lambda^{\frac{p_+ - p (\cdot)}{q
    (\cdot) - p (\cdot)} q (\cdot)} \leqslant \lambda^{q (\cdot)} \leqslant \;
    \lambda^{\kappa} & \text{on }  D \backslash E_{q (\cdot), \kappa}, \\
    \omega (\cdot)^{- | q (\cdot) - p_+ |} \; = \; \lambda^{\frac{q (\cdot) - p_+}{q
    (\cdot) - p (\cdot)} q (\cdot)} \leqslant \lambda^{q (\cdot)} \leqslant \;
    \lambda^{\kappa} & \text{on } D \backslash E_{q (\cdot), \kappa} . 
  \end{align*}
  Moreover, $\omega (\cdot)^{- | p_+ - p (\cdot) |} = \omega (\cdot)^{- | q (\cdot) - p_+ |} =
  1 \leqslant \lambda^{\kappa}$ on the set $D \cap E_{q (\cdot), \kappa}$ and
  therefore {\eqref{eq:condsonomega2}} holds. 

Finally, to prove  {\eqref{eq:condsonomega}} we estimate as follows:
  \begin{equation*}
    \rho_{p (\cdot), D} (\omega) = \int_{D \backslash E_{q (\cdot), \kappa}}
    \lambda^{- r (x)} d_{} x + | E_{q (\cdot), \kappa} | \; \leqslant \;
    \int_{D}
    \lambda^{- r (x)} d_{} x + | E_{q (\cdot), \kappa} | < \infty .
  \end{equation*}

\bigskip

\noindent   $[(i_{} v) \Rightarrow (i)]$ 
The proof of this implication is similar to the proof of sufficiency
in the proof of Theorem~\ref{primissimo}.  However, since $|\Omega|=\infty$ we
need to introduce $\omega$ and use $\rho_{\pp,D}(\omega)$ in place of
$|\Omega|$.  

As before, given a measurable function $f$ and a measurable set $E
\subseteq \Omega$, define 
\[ I (E) = \int_E M_{} f (x)^{p (x)} d_{} x, \qquad F(E) = \int_E
  |f(x)|^{p_+}\,dx. \]
Recall that $D = \{ x \in \Omega : p (x) < q (x) \}$ and write 
\[ \int_\Omega Mf(x)^{p(x)}\,dx = I(D) + I(\Omega \setminus D). \]
Since $p_+ \leqslant q_-$, we have
$p (\cdot) = p_+ = q_- = q (\cdot)$ on $\Omega \setminus
D$. Therefore, since $\pp\not\equiv 1$,  $p_+>1$, so the maximal operator is bounded on
$L^{p_+}(\Omega)$.  Hence,
\begin{equation*}
  I (\Omega \setminus D) = \int_{\Omega \setminus D} M f (x)^{p_+} dx
  \leqslant c_{p^+, n} \, F (\Omega) 
\end{equation*}

To estimate $I (D)$, define
$D_{\omega} (M f) = \{ x \in D : Mf(x) > \omega (x) \}$ where
$\omega$ is the function from  our hypothesis $(iv)$.  Then
\begin{align*}
  I (D) & =  \int_{D \setminus D_{\omega} (M_{} f)} M_{} f (x)^{p (x)} d_{}
  x + \int_{D_{\omega} (M f)} M f (x)^{p (x)} d x \\
  & \leqslant \rho_{p (\cdot), D} (\omega) + \int_{D_{\omega} (M_{} f)}
  \left( \frac{M_{} f (x)}{\omega (x)} \right)^{p (x)} \; \omega (x)^{p (x)}
  d_{} x.
\intertext{Since $M_{} f (\cdot) / \omega (\cdot) > 1$ on $D_{\omega} (M_{} f)$,}
 & \leqslant  \rho_{p (\cdot), D} (\omega) + \int_{D_{\omega} (M_{}
  f_2)} \left( \frac{M_{} f (x)}{\omega (x)} \right)^{p_+} \; \omega (x)^{p
  (x)} d_{} x \\
  & \leqslant  \rho_{p (\cdot), D} (\omega) + \| \omega^{- | p_+ - p (\cdot)
  |} \|_{L^{\infty} (D)} \int_D (M_{} f (x))^{p_+} d_{} x .
\intertext{Again since $M$ is bounded on $L^{p_+}(\Omega)$,}
& \leqslant  \rho_{p (\cdot), D} (\omega) + c_{n, p^+} \cdot \|
  \omega^{- | p_+ - p (\cdot) |} \|_{L^{\infty} (D)} \, F (\Omega).
\end{align*}

If we combine the above inequalities we get 
\begin{equation}
  I (\Omega) \leqslant \left[ c_{n, p^+} (1 + \| \omega^{- | p_+ - p (\cdot)
  |} \|_{L^{\infty} (D)})_{_{_{_{_{_{}}}}}} \right] F (\Omega) + \rho_{p
  (\cdot), D} (\omega), \label{eq:fundestimateI}
\end{equation}
so to complete the proof we need to  estimate $F (\Omega)=F (D) + F
(\Omega \setminus D)$. As
before we have $p (\cdot) = p_+ =
q_- = q (\cdot)$ on $\Omega \setminus D$, so 
\begin{equation*}
  F (\Omega \setminus D) = \int_{\Omega \setminus D} | f (x) |^{p_+} =
  \int_{\Omega \setminus D} | f (x) |^{q (x)} d_{} x. \label{eq:F2OmmD}
\end{equation*}

To estimate $F (D)$,  let $D_{\omega} (f) = \{ x \in D_{} :
|f (x) | > \omega (x) \}$.  Since $0 < \omega \leqslant 1$
and $p_+ \geqslant p (\cdot)$, we have $\rho_{p_+, D} (\omega) \leqslant
\rho_{p (\cdot), D} (\omega)$.  Therefore,
\begin{align*}
  F (D) & =  \int_{D \setminus D_{\omega} (f)} | f (x) |^{p_+} d_{} x +
  \int_{D_{\omega} (f)} | f (x) |^{p_+} d_{} x \\
  & \leqslant  \rho_{p (\cdot), D} (\omega) + \int_{D_{\omega} (_{} f)}
  \left( \frac{| f (x) |}{\omega (x)} \right)^{p_+} \cdot \omega (x)^{p_+}
  d_{} x. 
\intertext{Since $_{} | f (\cdot) | / \omega (\cdot) > 1$ on $D_{\omega} (_{} f)$}
 & \leqslant  \rho_{p (\cdot), D} (\omega) + \int_{D_{\omega} (_{} f)}
  \left( \frac{| f (x) |}{\omega (x)} \right)^{q (x)} \cdot \omega (x)^{p_+}
  d_{} x \nonumber\\
  & \leqslant  \rho_{p (\cdot), D} (\omega) + \| \omega^{- | q (\cdot) - p_+
  |} \|_{L^{\infty} (D)} \int_D | f (x) |^{q (x)} d_{} x. 
\end{align*}

If we combine the previous two estimates, we get
\begin{align*}
  F (\Omega) & \leqslant  \int_{\Omega \setminus D} | f (x) |^{q (x)} d_{} x
  + \rho_{p (\cdot), D} (\omega) + \| \omega^{- | q (\cdot) - p_+ |}
  \|_{L^{\infty} (D)} \int_D | f (x) |^{q (x)} d_{} x \\
  & \leqslant  (1 + \| \omega^{- | q (\cdot) - p_+ |} \|_{L^{\infty} (D)})
  \int_{\Omega} | f (x) |^{q (x)} d_{} x + \rho_{p (\cdot), D} (\omega) . 
  \label{eq:F2D}
\end{align*}
Together with inequality~\eqref{eq:fundestimateI} this gives us the
modular inequality~\eqref{mainmodular}.  This completes the proof.

\end{document}